\documentclass[10pt]{amsart}

\usepackage[latin9]{inputenc}
\usepackage{color}
\usepackage{amstext}
\usepackage{amsthm}
\usepackage{amssymb}
\usepackage[numbers,sort]{natbib}
\PassOptionsToPackage{normalem}{ulem}
\usepackage{ulem}

\makeatletter
\numberwithin{equation}{section}
\numberwithin{figure}{section}

\newtheorem{thm}{Theorem}[section]
\newtheorem{lemma}[thm]{Lemma}
\newtheorem{prop}[thm]{Proposition}

\newtheorem{rem}[thm]{Remark}

\numberwithin{equation}{section}

\newcommand{\les}{\lesssim}
\newcommand{\lam}{{\lambda}}

\newcommand{\ve}{{\varepsilon}}

\newcommand{\R}{{\mathbb R}}
\newcommand{\Z}{{\mathbb Z}}

\newcommand{\braxi}{\langle\xi\rangle}

\def\normo#1{\left\|#1\right\|}

\def\bra#1{{\langle#1\rangle}}

\global\long\def\R{\mathbf{\mathbb{R}}}%
\global\long\def\C{\mathbf{\mathbb{C}}}%
\global\long\def\Z{\mathbf{\mathbb{Z}}}%
\global\long\def\N{\mathbf{\mathbb{N}}}%
\global\long\def\tilde#1{\widetilde{#1}}%
\global\long\def\jp#1{\langle#1\rangle}%
\global\long\def\tilde#1{\widetilde{#1}}%
\global\long\def\ve{\varepsilon}%
\global\long\def\ep{\epsilon}%

\makeatother

\begin{document}
\title[Scattering for 2d semi-relativistic equations]{Scattering for 2d semi-relativistic Hartree equations with short range potential}
\author{Changhun Yang}
\address{Department of Mathematics, Chungbuk National University, Chungdae-ro1,
Seowon-gu, Cheongju-si, Chungcheongbuk-do, Republic of Korea}
\email{chyang@chungbuk.ac.kr}
\begin{abstract}
We study the long time behavior of small solutions to the semi-relativistic Hartree equations in two dimension. The nonlinear term is convolved with the singular potential $|x|^{-\gamma}$ for $1<\gamma<2$, which is referred to as \textit{short-range interaction} potential in the sense of scattering phenomenon. We establish the scattering results for small solutions  in a weighted space, in other words, we prove that the nonlinear solutions exist globally and behave asymptotically like a linear solution whenever the initial data is sufficiently small.   
To achieve this, we should obtain time decay estimates for the nonlinear term which is integrable. A key observation is that the loss in time in the course of weighted energy estimates can be recovered by the \textit{space resonance method} with the help of \textit{null structure}.
\end{abstract}

\maketitle

\section{Introduction}
We consider the semi-relativistic Hartree equations 
\begin{align}
\left\{ \begin{aligned}-i\partial_{t}u+\sqrt{m^{2}-\Delta}u & =\lam\left(|x|^{-\gamma}*|u|^{2}\right)u &&\mathrm{in}\;\;\R\times\mathbb{R}^{d},\\
u(0,\cdot) & =u_{0} &&\mathrm{in}\;\;\mathbb{R}^{d},
\end{aligned}
\right.\label{main-eq:semi}
\end{align}
where the unknown $u:\mathbb{R}^{1+d}\to\mathbb{C}$, some fixed
constant $\lam\in\R$, and $0<\gamma<d$. The nonlocal differential operator $\sqrt{m^{2}-\Delta}$
is defined as a Fourier multiplier operator associated to the symbol
$\sqrt{m^{2}+|\xi|^{2}}$ and $\ast$ denotes the convolution in $\mathbb{R}^d$. 
$m>0$ represents the mass which will be normalized to $1$ by scaling throughout the paper.
For three dimensional case $d=3$, \eqref{main-eq:semi} with Coulomb potential ($\gamma=1$) is often referred to as \textit{Boson star equation} which describes the dynamics and collapse of relativistic Boson stars. It was rigorously derived as the mean-field limit of large systems of bosons. See \cite{Michelangeli2012,frjonlenz2007-nonlinearity,Lieb1987} and references therein. \eqref{main-eq:semi} with generalized potential in all dimension has been also widely studied by many researchers in the mathematical viewpoint.
The singular potential $|x|^{-\gamma}$ is called 
the \textit{long range}  (or, \textit{short range}) interaction when $0<\gamma\le1$ (or,  $1<\gamma<d$, respectively).
The mass and energy of solution to \eqref{main-eq:semi} are defined by  
\begin{align}
    M(u)(t)&= \|u(t)\|_{L^2(\R^d)}, \label{mass conservation}\\ 
    E(u)(t)&= \frac12\int_{\R^d}\overline{u}\sqrt{m^2-\Delta} u dx 
    + \frac{\lambda}{4}\int_{\R^d}\left(|x|^{-\gamma}\ast |u|^2\right) |u|^2 dx, \nonumber
\end{align}
respectively and conserved as time evolves. From the conservations laws, one sees that $H^\frac12$ is the energy space. Furthermore,
in massless case $m=0$, we have the scaling symmetry. If $u$ is solution to \eqref{main-eq:semi}, $u_\alpha$ defined by 
$u_\alpha(t,x)=\alpha^{\frac{1-\gamma+d}{2}}u(\alpha t,\alpha x)$ is also solution, so the scaling invariant space is $H^{\frac{\gamma-1}{2}}(\mathbb{R}^d)$. Thus \eqref{main-eq:semi} is often called $H^{\frac{\gamma-1}{2}}(\mathbb{R}^d)$-critical. 


There have been numerous results on global solutions to semi-relativistic
Hartree equations \eqref{main-eq:semi}. Lenzmann \cite{lenz2007} proved that the Boson star equation ($\gamma=1$ and $d=3$) was globally well-posed in the energy space $H^\frac12(\R^3)$.
Cho and Ozawa \cite{choz2006-siam} extended this result to \eqref{main-eq:semi} under the suitable assumption on $\gamma$ and $d$. Furthermore, they verified the asymptotic behavior of solutions when $d\ge3$ and $2<\gamma<d$ by showing the small data scattering in $H^s(\R^d)$ for $s>\frac{\gamma}{2}-\frac{d-2}{2d}$. By a \textit{scattering} in $H^s(\R^d)$, we mean that solutions to nonlinear PDEs converge to solutions to the linear equation in $H^s(\R^d)$ as time goes to infinity. On the other hand, they computed the time decay of $L^2$ norm of the nonlinearity, computed on a solution of the linear equation, as follows:
\begin{align}\label{time decay of nonlinear term}
    \normo{\Big(|x|^{-\gamma}*\big|e^{-it\sqrt{m^2-\Delta}}u_0\big|^{2}\Big)e^{-it\sqrt{m^2-\Delta}}u_0}_{L^{2}(\R^d)}\sim |t|^{-\gamma} \quad \text{ for } |t|\gg 1.
\end{align}
The computation is based on the well-known dispersive property of the linear solutions  (see \cite[Lemma~3]{MNO2003})
\begin{align*}
    \| e^{-it\sqrt{m^2-\Delta}} u_0 \|_{L^\infty(\R^d)} \les (1+|t|)^{-\frac d2} \quad \text{for} \quad u_0\in C_0^\infty(\R^d).
\end{align*}
Based on \eqref{time decay of nonlinear term}, they further proved the nonexistence of scattering state when the \textit{long range} interaction $(0<\gamma\le1 \text{ for } d\ge3 \text{ and } 0<\gamma<1$ \text{ for } $d=2)$ was concerned where the time decay of nonlinear term fails to be integrable. In contrast, the global solutions to \eqref{main-eq:semi} with the \textit{short range} interaction  $(1<\gamma<d)$ would be expected to scatter since the nonlinear term decays sufficiently fast to be integrable.  For $d=3$, Pusateri \cite{pusa} confirmed this anticipation for sufficiently small initial data in a weighted spaces by closing the gap corresponding to $1<\gamma\le2$ by proving the scattering result of \eqref{main-eq:semi}. Inspired by the work \cite{pusa}, we focus on the two dimensional case with full \textit{short range} interaction, $1<\gamma<2$.

Besides, there have been also a lot of results on local well-posedness of \eqref{main-eq:semi} given rough initial data and global well-posedness under the various assumptions. We refer to \cite{choz2007-jkms,choz2008-dcds-s,chozhishim2009-dcds,hele2014,lee2021-bkms} and references therein for interested readers.

\begin{rem}\textup{ Concerning 
\eqref{main-eq:semi} with $\gamma=1$ which corresponds to the \textit{scattering-critical} equation where the decay of nonlinear term barely fails to be integrable, the modified scattering results have been established by Pusateri \cite{pusa} for three dimensional case $d=3$ and by Kwon, Lee, and the author of the paper \cite{kwon2023modified} for two dimensional case $d=2$.}
\end{rem}

\begin{rem}\textup{
    Let us introduce the fractional Schr\"odinger Hartree equation
    \begin{align}\label{eq:schrodinger} 
        i\partial_{t}u-(-\Delta)^{\alpha} u =\lam\left(|x|^{-\gamma}*|u|^{2}\right)u,\quad \mathrm{in}\;\;\R\times\mathbb{R}^{d},
    \end{align}
    where $u:\R^{1+d}\rightarrow\C$, $0<\alpha<1$, $\alpha\neq1$, $0<\gamma<d$ and $\lambda\in \R\setminus\{0\}$. The fractional Laplacian $(-\Delta)^\alpha$ is the nonlocal differential operator associated to the symbol $|\xi|^\alpha$.
    Cho \cite{Cho2017} established the small data scattering result for \eqref{eq:schrodinger} with the short range interaction in a weighted space when $d\ge3$. We expect that the argument in this paper can be applied to \eqref{eq:schrodinger} with two dimension.
}\end{rem}


We now state our main theorem for the two dimensional semi-relativistic Hartree equations \eqref{main-eq:semi} with the short range interaction.
\begin{thm}\label{appen-mainthm}
Let $s\ge 30$ and $1<\gamma<2$. There exists $\ep_0>0$ satisfying the following: 

If the initial data $u_0$ satisfies
\[
	\|u_0\|_{H^{s}(\R^{2})} + \|\langle x\rangle^2u_0\|_{H^{10}(\R^{2})} \le \ve,
	\]
for some $0<\ep\le \ep_0$, there exists a unique global solutions to \eqref{main-eq:semi} such that 
\[
	\sup_{t\in\R}\left( \|u(t)\|_{H^{s}(\R^{2})} + \|\langle x\rangle^2e^{it\langle D\rangle} u(t)\|_{H^{10}(\R^{2})} \right) \les \ve,
	\]
and 
\begin{align}\label{time decay of nonlinear sol}
	\|u(t)\|_{W^{7,\infty}(\R^2)} \les  (1+|t|)^{-1} \ve.    
\end{align}
Furthermore, there exist $\delta>0$ and asymptotic state $v_+$ 
with 
\begin{align*}
    \|v_+\|_{H^{s}(\R^{2})} + \|\langle x\rangle^2v_+\|_{H^{10}(\R^{2})} \les \ve,
\end{align*}
satisfying 
\begin{align*}
    \Big\| e^{it\sqrt{m^2-\Delta}}u(t,x) - v_+  \Big\|_{H^{s}(\R^{2})} +  \Big\|\langle x\rangle^2\Big( e^{it\sqrt{m^2-\Delta}}u(t,x) - v_+\Big)  \Big\|_{H^{10}(\R^{2})} 
     \lesssim (1+t)^{-\delta} \ve.
\end{align*}
\end{thm}

\begin{rem}\textup{
The required regularity $s$ depends on $\gamma$ but we are not interested in the dependence as well as their optimality in this paper.}
\end{rem}

In Theorem~\ref{appen-mainthm}, we prove the global existence of solutions to \eqref{main-eq:semi} when the initial data is sufficiently small in a weighted space. Also, we show that the small nonlinear solutions decays as time evolves with the optimal decay rate in the sense that they decay with the same rate as the linear solution.  Furthermore, the solution scatters, i.e., behaves like a linear solution in the weighted spaces as time goes to infinity with the convergence rate $t^{-\delta}$.

Our proof is based on the standard bootstrap argument in a weighted Sobolev spaces. We adapt the weighted space so that we obtain and utilize the optimal time decay of solutions from the well-known dispersive estimates (see \eqref{Freq localized dispersive estiamtes}). Then, we
perform the weighted energy estimates to prove the bounds of the energy norms under the a priori assumption. The strategy is based on the method of \textit{space-time resonance} which was developed by Germain-Masmoudi-Shata \cite{gemasha2008,gemasha2012-jmpa,gemasha2012-annals}.

Let us explain in detail the ideas of proof in  the weighted energy estimates. We introduce the interaction representation of solutions $u(t)$ so as to track the scattering states 
\begin{align*}
f(t,x):=e^{it\sqrt{m^2-\Delta}}u(t,x).
\end{align*}
Then we can express $f$ via Duhamel's representation
\begin{align}\label{eq:duhamel}
\begin{aligned}\widehat{f}(t,\xi) & =\widehat{u_{0}}(\xi)+i\lam\mathcal{I}(t,\xi),\\
\mathcal{I}(t,\xi) & =\int_{0}^{t}\int_{\mathbb{R}^{2}}e^{it'\phi(\xi,\eta)}|\eta|^{-2+\gamma}\widehat{f}(t',\xi-\eta)\widehat{|u|^2}(t',\eta)d\eta dt'
\end{aligned}
\end{align}
with the resonance function 
\begin{align*}
\phi (\xi,\eta)=\braxi-\langle\xi-\eta\rangle.
\end{align*}
In the course of weighted energy estimates, we have to bound $x^2f$ in $H^{10}(\R^2)$ which is converted to $\nabla^2 \widehat{f}$ in the fourier space. The most delicate term in $\nabla^2 \widehat{f}$  occurs when two derivatives both fall on $e^{it'\phi(\xi,\eta)}$ to give us 
\begin{align}\label{delicate term}
\int_{0}^{t}\int_{\mathbb{R}^{2}}|t'|^2\big( \nabla_\xi \phi(\xi,\eta) \big)^2 e^{is\phi(\xi,\eta)}|\eta|^{-2+\gamma}\widehat{f}(t',\xi-\eta)\widehat{|u|^2}(t',\eta)d\eta dt',
\end{align}
where we have to compensate the loss in time, $|t'|^2$. 
Here, we encounter the main difficulty from two dimensional nature, i.e., the weaker time decay $|t'|^{-1}$ of solutions \eqref{time decay of nonlinear sol} compared to three or higher dimensional problem because if we just take $L^2$-norm of cubic nonlinearity in \eqref{delicate term}, we only manage to obtain time decay $|t'|^{-2}$ at best which is not sufficient to not only recover the time loss but also for the nonlinear term to be integrable in time.
A key observation to overcome this is to exploit the structure in the quadratic term $|u|^2$, in other words, to exploit the \textit{space resonance}. Writing it it in terms of $f$ as 
\begin{align}\label{FT |u|^2 2}
    \widehat{|u|^2 }  (t',\eta)
    = \frac{1}{(2\pi)^2}\int_{\R^2} e^{-it'\phi(\sigma,\eta)}\widehat{f}(t',\sigma-\eta)\overline{\widehat{f}(t',\sigma}) d\sigma,
\end{align}
we find that whenever we preform an integration by parts by using the following relation 
\begin{align*}
    e^{it' \phi} = - \frac{i}{t'} \, \frac{\nabla_\sigma \phi \cdot \nabla_\sigma e^{it' \phi}}{|\nabla_\sigma \phi|^2},    
\end{align*}
we can obtain a time decay $|t'|^{-1}$.
At the same time, however, it yields a singularity in $\eta$ variable because 
\begin{align*}
    \nabla_\sigma\phi(\sigma,\eta)|_{\eta=0}=\nabla_\sigma \Big( \langle\sigma\rangle-\langle\sigma-\eta\rangle \Big) \Big|_{\eta=0} = 0.
\end{align*}
Nevertheless, in order to recover the time loss, $|t'|^2$, in \eqref{delicate term}, we indeed integrate by parts \eqref{FT |u|^2 2} twice, which will be possible because we a priori assume $x^2f \in H^{10}(\R^2)$.

The resulting multiplier which is singular near the origin, or more specifically comparable to $|\eta|^{-2}$, can be bounded with the help of \textit{null structure} in \eqref{delicate term}. Indeed, as already observed above, we have \begin{align*}
    \nabla_\xi \phi(\xi,\eta) |_{\eta=0} = 0,
\end{align*}
so we can think that the following multiplier behaves like the constant
$$|\nabla_\xi \phi(\xi,\eta)|^2 |\nabla_\sigma \phi(\sigma,\eta)|^{-2} \approx 1,$$
as far as estimates are concerned\footnote{See \eqref{lowbound} for precise inequality.}.
Thus, we can finally estimate \eqref{delicate term} similarly to \eqref{eq:duhamel} under the assumption $x^2f\in H^{10}(\R^2)$. We remark that the regularity of weighted function is required when we apply the multiplier inequality in this procedure.

\smallskip

\subsection*{Organization} In Section~2, we provide the proof of main theorem relying on the standard bootstrapping argument. The rest sections are devoted to proving the weighted energy estimates which played a key role in Section~2. In section~3, we discuss some auxiliary estimates and establish the time decay estimates for the frequently localized quadratic term by exploiting the space resonance. In section~4, we finally prove the weighted energy estimates with the help of null structures.

\smallskip

\subsection*{Notations}
\; 

\noindent $\bullet$ (Fourier transform)
$\mathcal F g(\xi)=\widehat{g}(\xi):=\int_{\R^2}e^{-ix\cdot\xi}g(x)dx$ and  $g(x)=\frac{1}{(2\pi)^2}\int_{\R^2}e^{ix\cdot\xi}\widehat{g}(\xi)d\xi$.

\noindent $\bullet$ (Japaneses bracket)
$\langle x\rangle:=(1+|x|^{2})^{\frac{1}{2}}$ for $x\in\mathbb{R}^{2}$.

\noindent $\bullet$ (Mixed-normed space) For a Banach space $X$
and an interval $I$, $u\in L_{I}^{q}X$ iff $u(t)\in X$ for a.e. $t\in I$
and $\|u\|_{L_{I}^{q}X}:=\|\|u(t)\|_{X}\|_{L_{I}^{q}}<\infty$. Especially,
we denote $L_{I}^{q}L_{x}^{r}=L_{t}^{q}(I;L_{x}^{r}(\mathbb{R}^2))$, $L_{I,x}^{q}=L_{I}^{q}L_{x}^{q}$,
$L_{t}^{q}L_{x}^{r}=L_{\mathbb{R}}^{q}L_{x}^{r}$.

\noindent $\bullet$ As usual different positive constants are denoted by the same letter $C$, if not specified. $A\lesssim B$ and $A\gtrsim B$ means that
$A\le CB$ and $A\ge C^{-1}B$, respectively for some $C>0$. $A\sim B$
means that $A\lesssim B$ and $A\gtrsim B$.

\noindent $\bullet$ (Fourier multiplier) $D=-i\nabla$. For $m:\R^2\rightarrow\R$, $m(D)f:=\mathcal{F}^{-1}\big( m(\xi)\widehat{f}(\xi) \big)$.

\noindent $\bullet$ (Littlewood-Paley operator) Let $\chi$ be a
bump function such that $\chi\in C_{0}^{\infty}(B(0,2))$
with $\chi(\xi)=1$ for $|\xi|\le1$ and define $$\chi_{L}(\xi):=\chi\big(\tfrac{\xi}{L}\big)-\chi\big(\tfrac{2\xi}{L}\big)
\text{ for } L\in2^{\mathbb{Z}}.$$
We define the projection operator $P_{L}$ for $L\in2^{\Z}$ via $
    \mathcal{F}(P_{L}f)(\xi)=\chi_L(\xi)\widehat{f}(\xi)$.
Then, we immediately have $ f= \sum_{L\in 2^{\Z}}P_Lf$.

We also introduce an another partition of unity $\rho_N$ for $N\in 2^{\N\cup\{0\}}$ such that 
\begin{align*}
    \rho_1= \chi \text{ and } \rho_N=\chi_N \text{ for } N\ge 2.
\end{align*}
We define the projection operator $S_{N}$ for $N\in 2^{\N\cup\{0\}}$ by 
$\mathcal{F}(S_{N}f)(\xi)=\rho_N(\xi)\widehat{f}(\xi)$. 
So, $S_1=\sum_{L \le 1}P_L$ and $S_N=P_N$ for $N\ge2$.
By orthogonality, we have 
\begin{align*}
    \| f\|_{H^s}^2   \approx \sum_{L\in 2^{\Z}} \langle L \rangle^{2s} \| P_{L}f\|_{L^{2}}^{2}\approx \sum_{N\in 2^{\N}} N^{2s}\| S_Nf\|_{L^2}^2.
\end{align*}

\noindent $\bullet$ (Commutator) The commutator of $A$ and $B$, denoted by $[A,B]$, is defined as $[A,B]=AB-BA$.

\noindent $\bullet$ For $\alpha=(\alpha_1,\alpha_2)$, where $\alpha_i\in \N\cup\{0\}$ for $i=1,2$, we define $\partial_x^{\,\alpha}=\partial_{x_1}^{\,\alpha_1}\partial_{x_2}^{\,\alpha_2}$ and $|\alpha|=\alpha_1+\alpha_2$.

 \noindent $\bullet$ For the distinction between a vector and a scalar,
 we use the bold letter for a vector-valued function and the normal
 letter for a scalar-valued function.

\noindent $\bullet$ Let $\textbf{v}=(v_{1},v_{2}),\textbf{w}=(w_{1},w_{2})\in\C^{2}$.
Then $\textbf{v}\otimes\textbf{w}$ denotes the usual tensor product
such that $(\textbf{v}\otimes\textbf{w})_{ij}=v_{i}w_{j}$ for $i,j=1,2$. 
We also denote a tensor product of $\textbf{v}\in\C^{n}$ and $\textbf{w}\in\C^{m}$ by a matrix $\textbf{v}\otimes\textbf{w}=(v_{i}w_{j})_{\substack{i=1,\cdots,n\\
i=1,\cdots,m}}$. 
For simplicity, we use the simplified notation 
\[
\mathbf{v}^{k}=\overbrace{\mathbf{v}\otimes\cdots\otimes\mathbf{v}}^{k\;\text{times}},\qquad\nabla^{k}=\overbrace{\nabla\otimes\cdots\otimes\nabla}^{k\;\text{times}}.
\]
The product of $\mathbf{v}$ and $f\in\C$ is given by $\mathbf{v}f=\mathbf{v}\otimes f$.

\subsection*{Acknowledgement}
 The author was supported by the National Research Foundation of Korea(NRF) grant funded by the Korea government(MSIT) (No. 2021R1C1C1005700).


\section{Proof of Theorem~\ref{appen-mainthm}}
Throughout the paper, we fix $m=1$ by scaling.
In this section, we prove the Theorem~\ref{appen-mainthm}.
Let $u(t)$ be a local solution on a time interval $[0,T]$, i.e. $u\in C([0,T]:H^s(\R^2))$, which can be constructed by the standard fixed point argument.\footnote{For example, see \cite[Proposition~2.1]{choz2006-siam}.} 
For given a solution to \eqref{main-eq:semi}, let us define  
\begin{align*}
    f(t,x):= e^{it\sqrt{1-\Delta}} u(t,x).
\end{align*}
Using the Duhamel's formula, we write \eqref{main-eq:semi} as 
\begin{align}\label{main:integral form}
    f(t) = u_0 +i\lambda\int_0^t e^{it'\sqrt{1-\Delta}} N_\gamma(u,u,u)(t')dt', 
\end{align}
where 
\begin{align*}
    N_\gamma(u,v,w):= \big(|x|^{-\gamma}\ast(u\overline{v})\big) w, \qquad 1<\gamma<2.
\end{align*}
We solve the integral equation \eqref{main:integral form} in the space defined via the following norm 
\begin{align*}
	\sup_{t\in \R}\left( \|u(t)\|_{H^{s}(\R^2)}  + \|\langle x\rangle^2f(t)\|_{H^{10}(\R^2)} \right).
\end{align*}
Given a local solution $u$ on a time interval $[0,T]$, we assume that the following norm is a priori small 
\begin{align}\label{appen-apriori}
	\|u\|_{X_T}:=\sup_{t\in[0,T]}\left( \|u(t)\|_{H^{s}(\R^2)}  + \|\langle x\rangle^2f(t)\|_{H^{10}(\R^2)} \right) \le \ve_1,
\end{align}
for some $\ve_1>0$.
In order to prove the existence of global solution, by the standard bootstrapping argument, we suffice to show that under the a priori assumptions, one has 
\begin{align}\label{ineq:bootstrapping}
    \|u\|_{X_T} \le \ve_0 + C\ve_1^3,
\end{align}
where $\ve_0$ is the size of the initial data 
\[
	\|u_0\|_{H^{s}(\R^{2})} + \|\langle x\rangle^2u_0\|_{H^{10}(\R^{2})} \le \ve_0.
	\]

Let us recall from \cite[Lemma~3]{MNO2003} the following frequency localized dispersive estimates 
\begin{align}\label{Freq localized dispersive estiamtes}
    \| e^{it\langle D\rangle} S_N \phi\|_{L^\infty(\R^2)}
    &\lesssim \langle t\rangle^{-1} N^2 \|  \phi\|_{L^1(\R^2)} \quad \text{ for } N\ge1.
\end{align}
Then, one verifies that under the a priori assumption, our solution $u$ satisfies
\begin{align}\label{time decay}
    \| u(t) \|_{W^{7,\infty}(\R^2)} \lesssim \langle t\rangle^{-1}\ve_1.
\end{align}
Indeed, using the dispersive estimates \eqref{Freq localized dispersive estiamtes}, we estimate 
\begin{align*}
 \|\langle D\rangle^7 & u(t)\|_{L^\infty(\R^2)}  \le \sum_{N\ge1} \langle N\rangle^7 \| e^{it\langle D\rangle} S_N f(t)\|_{L^\infty(\R^2)}
\\  &\lesssim \langle t\rangle^{-1} \sum_{N\ge1} \langle N\rangle^9 \|  \langle x\rangle^2 S_N f(t)\|_{L^2(\R^2)} \lesssim \langle t\rangle^{-1}\big( \|x f(t)\|_{H^{9}(\R^2)}  + \|\langle x\rangle^2 f(t)\|_{H^{10}(\R^2)} \big),
\end{align*}
where we bounded the weighted norm as follows
\begin{align*}
    \|  \langle x\rangle^2 S_N f\|_{L^2(\R^2)}
    &\le  \| [\langle x \rangle^{2},S_{N}]f \|_{L^2(\R^2)} + \| S_{N}  \langle x \rangle^{2}f \|_{L^2(\R^2)} \\ 
    &\lesssim  N^{-1}\|S_{\sim N}xf\|_{L^{2}(\R^2)}+ \| S_{N} \langle x \rangle^{2}f \|_{L^2(\R^2)},
\end{align*}
where $S_{\sim N}$ denotes $\displaystyle\sum_{\{M\in2^{\N\cup\{0\}}\,:\,\frac{N}{4}\le M\le 4N\}}S_{M}$.
Then, 
\begin{align}\begin{aligned}\label{x by x square}
    \|x f(t)\|_{H^{9}(\R^2)}  
    &\approx \sum_{|\alpha|\le 9} \| \partial^\alpha(xf) \|_{L^2(\R^2)} \\ 
    &\lesssim \sum_{|\alpha|\le 9} \sum_{\beta+\beta'=\alpha}\| \partial^\beta(x\langle x\rangle^{-2})\|_{L^\infty(\R^2)}\big\|\partial^{\beta'}\big(\langle x\rangle^2 f\big) \big\|_{L^2(\R^2)} \lesssim  \|\langle x\rangle^2 f\|_{H^{9}(\R^2)}.
\end{aligned}\end{align}

Using the time decay of solutions under the a priori assumption, we compute the time decay of Hartree nonlinear terms.

\begin{prop}[Weighted Energy Estimates]\label{Prop:Weighted estimates}
    Assume that $u\in C([0,T]:H^{s}(\R^2))$ satisfies the a priori assumptions \eqref{appen-apriori}. Then, we have
    \begin{align}
        \left\| N_\gamma(u,u,u)(t)\right\|_{H^{s}(\R^2)} &\lesssim (1+|t|)^{-\gamma}\ve_1^3, \label{weighted estimates 1} \\ 
        \left\| x^2 e^{it\langle D\rangle} N_\gamma(u,u,u)(t) \right\|_{H^{10}(\R^2)}
        &\lesssim (1+|t|)^{-1-\frac{\gamma-1}{3}}\ve_1^3. \label{weighted estimates 3}  
    \end{align}
    \end{prop}
The poof of Proposition~\ref{Prop:Weighted estimates} will be provided in Section~\ref{sec:Weighted-Energy-estimate}. 
One sees that the time decay of nonlinear term is integrable, which immediately implies that 
\begin{align*}
    \sup_{t\in[0,T]} \left\| \int_0^t e^{it'\langle D\rangle} N_\gamma(u,u,u)(t')dt' \right\|_{H^{s}(\R^2)} &\lesssim \ep_1^3, \\ 
       \sup_{t\in[0,T]} \left\| \langle x\rangle^2\int_0^t e^{it'\langle D\rangle} N_\gamma(u,u,u)(t')dt' \right\|_{H^{10}(\R^2)} &\lesssim \ep_1^3,
\end{align*}
which completes the proof of \eqref{ineq:bootstrapping}.
Furthermore, once we set a linear profile $v_+$ as  
\begin{align*} 
	v_+\, := \lim_{t\to\infty} f(t)
    = u_0 +i\lambda \lim_{t\to\infty} \int_0^t e^{it'\langle D\rangle} N_\gamma(u,u,u)(t')dt',
\end{align*}
we have from \eqref{weighted estimates 1} 
\begin{align*}
    \left\| e^{it\langle D\rangle}u(t,x) - v_+ \right\|_{H^{s}(\R^2)} \lesssim 
    \left\| \int_t^\infty  e^{it'\langle D\rangle} N_\gamma(u,u,u)(t')dt' \right\|_{H^{s}(\R^2)} \lesssim (1+t)^{-\gamma+1} \ve_1^3,
\end{align*}
and similarly, from \eqref{weighted estimates 3} 
\begin{align*}
    \left\| \langle x\rangle^2 \left(e^{it\langle D\rangle}u(t,x) - v_+\right)\right\|_{H^{10}(\R^2)}\lesssim (1+t)^{-\frac{\gamma-1}{3}} \ve_1^3.
\end{align*}


\section{Time decay estimates}
In this section, we establish the time decay estimates for the nonlinear terms and frequency localized solutions under the a priori assumption. 
\subsection{Preliminaries}
We begin with estimating quadratic terms convolved with the short range potential.
\begin{lemma}
    Let $1<\gamma<2$. For any $\mathbb{C}$-valued functions $u\in L^{2}(\R^{2})\cap L^{\infty}(\R^{2})$,
    we get 
    \begin{align}\label{eq:hls}
    \left\||x|^{-\gamma}*|u|^{2}\right\|_{L^{\infty}(\mathbb{R}^{2})} & \les\|u\|_{L^{2}}^{2-\gamma}\|u\|_{L^{\infty}}^{\gamma}.
    \end{align}
    \end{lemma} 
\begin{proof} The proof is quite standard, but we provide it for completeness. We estimate 
    \begin{align*}
        \| |x|^{-\gamma}\ast |u|^2 \|_{L^\infty(\R^2)}
        &=\sup\int_{\R^2}|y|^{-\gamma}|u|^2(x-y)dy\\
        &\le
        \sup\int_{|y|<R}|y|^{-\gamma}|u|^2(x-y)dy
        +\sup\int_{|y|\ge R}|y|^{-\gamma}|u|^2(x-y)dy \\ 
        &\le \|u\|_{L^\infty(\R^2)}^2 R^{2-\gamma} + R^{-\gamma}\|u\|_{L^2(\R^2)}^2.
    \end{align*}
Then taking $R=\|u\|_{L^\infty(\R^2)}^{-1}\|u\|_{L^2(\R^2)}$ gives \eqref{eq:hls}.    
\end{proof}

The next lemma is useful when we handle the Fourier multiplier operator.
\begin{lemma}[Coifman-Meyer operator estimates]\label{lem:coif}
Assume that a multiplier $\textbf{m}(\xi,\eta)$ satisfies that 
\begin{align}\label{multiplier bound}
    C(\textbf{m}):=\left\Vert \iint_{\mathbb{R}^{2}\times\R^{2}}\textbf{m} (\xi,\eta)e^{ix\cdot\xi}e^{iy\cdot\eta}\,d\eta d\xi\right\Vert _{L_{x,y}^{1}(\mathbb{R}^{2}\times\R^{2})}<\infty.    
\end{align}
Then, for $\frac{1}{p}+\frac{1}{q}=\frac{1}{2}$, 
\begin{align}\label{eq:coif-1}
\left\Vert \int_{\mathbb{R}^{2}}\textbf{m}(\xi,\eta)\widehat{u}(\xi\pm\eta)\widehat{v}(\eta)\,d\eta\right\Vert _{L_{\xi}^{2}(\R^2)}\les  C(\textbf{m})\|u\|_{L^{p}(\R^2)}\|v\|_{L^{q}(\R^2)}.
\end{align}
\end{lemma}
\begin{proof}
See, for example, \cite[Lemma~B.1]{pusa}.
\end{proof}

\smallskip 

\subsection{Time decay estimates}
We find time decay of the frequency localized quadratic terms under the a priori assumption.
\begin{lemma}\label{lem:quadratic terms} Let
    $u$ satisfy the a priori assumption \eqref{appen-apriori}. For a dyadic number $N\in2^{\N\cup\{0\}}$, we have 
    \begin{align}\label{eq:norm-infty}
    \left\Vert S_{N}\big(|u(t)|^{2}\big)\right\Vert _{L^{\infty}(\R^2)} 
    & \les \min\big(  N^{-7}\langle t\rangle^{-2}, N^{2-s}\big)\ve_{1}^{2},
    \end{align}
\end{lemma}
\begin{proof}
We write 
\begin{align*}
    S_N\big(|u(t)|^2\big)(x)&=\frac{1}{(2\pi)^2}\int_{\R^2}e^{ix\cdot\xi}\rho_N(\xi)\langle \xi\rangle^{-7}\mathcal{F}\big\{ \langle D\rangle^7 (|u(t)|^2)\big\} (\xi)d\xi \\ 
    &= \frac{1}{(2\pi)^2}\mathcal{F}^{-1} \big(\rho_N(\xi)\langle \xi\rangle^{-7}\big) \ast \langle D\rangle^{7} (|u(t)|^2)(x).
\end{align*}
By Young's inequality, we have 
\begin{align*}
    \left\Vert S_{N}\big(|u(t)|^{2}\big)\right\Vert _{L^{\infty}(\R^2)}
    &\lesssim \| \mathcal{F}^{-1} (\rho_N(\xi)\langle \xi\rangle^{-7})\|_{L^1(\R^2)} \|\langle D\rangle^{7} (|u(t)|^2)\|_{L^\infty(\R^2)} \\ 
    &\lesssim N^{-7} \|u(t)\|_{W^{7,\infty}(\R^2)}^2,
\end{align*}
which completes the proof of the first bound in \eqref{eq:norm-infty} by \eqref{time decay}.   

We consider the second bound. Using the Sobolev embedding, we immediately obtain that 
\begin{align*} 
   \left\Vert S_{N}\big(|u(t)|^{2}\big)\right\Vert _{L^\infty(\R^2)} & \les N\left\Vert S_{N}\big(|u(t)|^{2}\big)\right\Vert _{L^2(\R^2)}.
\end{align*}
The Plancherel's identity yields that 
\begin{align*}
    \| S_N \big(|u(t)|^{2}\big)\|_{L^2(\R^2)}
    = \frac{1}{2\pi}\| \rho_N \mathcal{F} \big( |u(t)|^2\big) \|_{L^2(\R^2)},
\end{align*}
where 
\begin{align*}
    \rho_N(\eta)\mathcal{F} \big( |u(t)|^2\big) (\eta) &= \int_{\R^2}\rho_N(\eta) \widehat{u}(\sigma)\overline{\widehat{u}(\sigma+\eta)} d\sigma \\
&=\int_{\R^2} \frac{\rho_N(\eta)}{\langle \sigma\rangle^s\langle\sigma+\eta\rangle^s}\widehat{\langle D\rangle^s u}(\sigma)\overline{\widehat{\langle D\rangle^su}(\sigma+\eta)} d\sigma, 
\end{align*}
which implies  that  \begin{align}\label{ineq:SNL2}
\| S_{N}|u(t)|^2 \|_{L^2(\R^2)} \lesssim  N\langle N\rangle^{-s}\| u\|_{H^s(\R^2)}^2. 
\end{align}\end{proof}


In view of \eqref{time decay} and \eqref{eq:norm-infty} in the previous lemma, it seems that time decay $\langle t\rangle^{-2}$ for $L^\infty(\R^2)$ norm of quadratic term is optimal. However, by exploiting structure from bilinear interaction, we can obtain an extra time decay under the a priori assumption at the cost of a derivative for the low frequency part.
In other words, we employ the space resonance method.
\begin{lemma} Let $u$ satisfy the a priori assumption \eqref{appen-apriori}. For a dyadic number $L\in2^{\Z}$, we have
\begin{align} 
    \left\Vert P_{L}\big(|u(t)|^{2}\big)\right\Vert _{L^2(\R^2)} & \les \min\big(   L^{-2}\bra{t}^{-3}, L\langle L\rangle^{-s} \big)\ve_{1}^{2}.\label{eq:norm-two} 
\end{align}
\end{lemma}
\begin{rem}
    By the Sobolev embedding, we immediately have that 
    \begin{align} 
       \left\Vert P_{L}\big(|u(t)|^{2}\big)\right\Vert _{L^\infty(\R^2)} & \les \min\big( L^{-1}\bra{t}^{-3} , L^2\langle L\rangle^{-s} \big)\ve_{1}^{2}.\label{eq:norm-infty 2} 
   \end{align}
\end{rem}
\begin{proof} 

The second bound in \eqref{eq:norm-two} can be proved similarly to the proof of \eqref{ineq:SNL2}.

We consider the first bound in \eqref{eq:norm-two}. We write 
\begin{align*}
    \mathcal{F} \big( |u(t)|^2\big) (\eta)
    = \frac{1}{(2\pi)^2}\int_{\R^2} e^{it(\bra{\sigma+\eta} -  \bra{\sigma})}\widehat{f}(t,\sigma)\overline{\widehat{f}(t,\eta+\sigma}) d\sigma,
\end{align*}
where $f(t,x)=e^{it\sqrt{1-\Delta}}u(t,x)$.
Using the relation 
\begin{align}\label{rel}
    e^{it (\bra{\sigma+\eta} -  \bra{\sigma})} = - \frac{i}{t} \, \frac{\nabla_\sigma (\bra{\sigma+\eta} -  \bra{\sigma}) \cdot \nabla_\sigma e^{it (\bra{\sigma+\eta} -  \bra{\sigma})}}{|\bra{\sigma+\eta} -  \bra{\sigma}|^2},    
\end{align}
we perform an integration by parts twice to write  
\begin{align*}
    \mathcal{F} \big( |u(t)|^2\big) (\eta) 
    = J_1+J_2+J_3+J_4,
\end{align*}
where 
\begin{align*}
J_1(t,\eta)&= \frac{1}{t^2}\int_{\R^2}\nabla_\sigma \cdot \big( \mathbf{m}(\eta,\sigma) \, \nabla_\sigma\cdot \mathbf{m}(\eta,\sigma)\big) 
\widehat{u}(t,\sigma)\overline{\widehat{u}(t,\eta+\sigma)} d\sigma,  \\ 
J_2(t,\eta)&=\frac{2}{t^2}\int_{\R^2}e^{it(\langle \sigma+\eta\rangle - \langle\sigma\rangle )} \nabla_\sigma\cdot \mathbf{m}(\eta,\sigma)\,\mathbf{m}(\eta,\sigma)\cdot \nabla_\sigma \Big( \widehat{f}(t,\sigma)\overline{\widehat{f}(t,\eta+\sigma})\Big) d\sigma, \\ 
J_3(t,\eta)&=\frac{1}{t^2}\int_{\R^2}e^{it(\langle \sigma+\eta\rangle - \langle\sigma\rangle )} \nabla_\sigma\otimes \mathbf{m}(\eta,\sigma)\cdot \mathbf{m}(\eta,\sigma)\otimes \nabla_\sigma \Big( \widehat{f}(t,\sigma)\overline{\widehat{f}(t,\eta+\sigma})\Big) d\sigma, \\ 
J_4(t,\eta)&=\frac{1}{t^2}\int_{\R^2}e^{it(\langle \sigma+\eta\rangle - \langle\sigma\rangle )} \mathbf{m}\otimes \mathbf{m}(\eta,\sigma)\cdot \nabla_\sigma \otimes \nabla_\sigma \Big( \widehat{f}(t,\sigma)\overline{\widehat{f}(t,\eta+\sigma})\Big) d\sigma,
\end{align*}
where 
\begin{align*}
    \mathbf{m}(\eta,\sigma) := \frac{\nabla_\sigma(\langle \sigma+\eta\rangle - \langle\sigma\rangle)}{|\nabla_\sigma(\langle \sigma+\eta\rangle - \langle\sigma\rangle)|^2}.
\end{align*}
Since
\begin{align}\label{lowbound}
    \frac{|\eta|}{\max(\langle \eta+\sigma\rangle, \langle\sigma\rangle) \min(\langle \eta+\sigma\rangle, \langle\sigma\rangle)^2}\lesssim  \left|\frac{\eta+\sigma}{\bra{\eta+\sigma}}-\frac{\sigma}{\bra{\sigma}}\right| \lesssim  \frac{|\eta|}{\max(\langle\eta+\sigma\rangle,\langle\sigma\rangle)},
\end{align}
one can easily verify that for multi index $\alpha$,
\begin{align}\label{derivative of m}
   \left|  \partial_\sigma^\alpha \mathbf{m}(\eta,\sigma) \right| \lesssim |\eta|^{-1}\max(\langle \eta+\sigma\rangle, \langle\sigma\rangle) \min(\langle \eta+\sigma\rangle, \langle\sigma\rangle)^{2+|\alpha|}.
\end{align}

First, we consider $J_1$. By dyadic decomposition, we write $J_1$ as 
\begin{align*}
    \chi_L(\eta) J_1(t,\eta) &= \sum_{N_1,N_2\ge1} J_1^{(L,N_1,N_2)}(t,\eta),  \\ 
    J_1^{(L,N_1,N_2)}(t,\eta)&= \frac{1}{t^2}\int_{\R^2} 
    m_1^{(L,N_1,N_2)}(\eta,\sigma) \widehat{S_{N_1}u}(t,\sigma)\overline{\widehat{S_{N_2}u}(t,\eta+\sigma)} d\sigma \\ 
    m_1^{(L,N_1,N_2)}(\eta,\sigma)&=\nabla_\sigma \cdot \big(  \mathbf{m}(\eta,\sigma) \, \nabla_\sigma\cdot  \mathbf{m}(\eta,\sigma)\big)\chi_L (\eta)\rho_{N_1} (\eta+\sigma)\rho_{N_2} (\sigma),
\end{align*}
where, by the support condition, the summation actually runs over 
\begin{align*}
    \mathcal A_L = \{ (N_1,N_2) \in 2^{\N \cup\{0\}} \; : \;  N_1\sim N_2\gtrsim L, \text{ or } \min(N_1,N_2)\ll \max(N_1,N_2) \sim L \}.
\end{align*}
Indeed, the integral $J_1^{(L,N_1,N_2)}$ is vanishing for $(N_{1},N_2)\notin \mathcal{A}_L$.

Using the operator inequality \eqref{eq:coif-1} with $C\big(m_1^{(L,N_1,N_2)}\big)$ satisfying
\begin{align*}
    C\big(m_1^{(L,N_1,N_2)}\big):= \left\Vert \iint_{\mathbb{R}^{2}\times\R^{2}}m_1^{(L,N_1,N_2)}(\eta,\sigma)e^{ix\cdot\eta}e^{iy\cdot\sigma}\,d\eta d\sigma\right\Vert _{L_{x,y}^{1}(\mathbb{R}^{2}\times\R^{2})}<\infty ,
\end{align*}
we have 
\begin{align*}
    \left\| J_1^{(L,N_1,N_2)}(t,\eta) \right\|_{L_{\eta}^2(\R^2)}  \les |t|^{-2} C\big(m_1^{(L,N_1,N_2)}\big)\normo{S_{N_{1}}{u}(t)\,}_{L^{2}(\R^2)}\normo{S_{N_2}|u(t)|^{2}}_{L^{\infty}(\R^2)}.
\end{align*}
Using \eqref{lowbound} and \eqref{derivative of m}, a direct computation yields that\footnote{\,For the detailed computation, we refer to \cite[Section~5.2]{CKLY2022}} 
\begin{align*}
  \big|m_1^{(L,N_1,N_2)}(\eta,\sigma)\big| \les L^{-2}\max(N_1, N_2)^2 \min(N_1,  N_2)^6 
\end{align*}
and 
\begin{align}\label{cmbound1}
    C\big({m_1^{(L,N_1,N_2)}}\big) \les L^{-2}\max(N_1, N_2)^2  \min(N_1,  N_2)^{14},
\end{align}
thus, by Lemma~\ref{lem:quadratic terms}, we have under the a priori assumption
\begin{align*}
    &\left\| \chi_L(\eta) J_1(t,\eta)\right\|_{L^2(\R^2)}\lesssim 
    \sum_{N_1,N_2\ge1}  \left\| J_1^{(L,N_1,N_2)}(t) \right\|_{L^2(\R^2)} \\ 
    &\;\les |t|^{-2}  L^{-2} \sum_{(N_1,N_2)\in \mathcal A_L}  \max(N_1, N_2)^2  \min(N_1,  N_2)^{14}
    \|S_{N_1}u(t)\|_{L^2(\R^2)} \|S_{N_2}u(t)\|_{L^\infty(\R^2)} \\ 
    &\;\les |t|^{-2}  L^{-2} \sum_{(N_1,N_2)\in \mathcal A_L} \max(N_1, N_2)^2 \min(N_1,  N_2)^{14}   N_1^{-s} 
   N_2^{-7}\langle t\rangle^{-1}\ve_1^2 \\ 
    &\;\les \langle t\rangle^{-3}L^{-2}\ve_1^2.
\end{align*}

Next, we consider $J_2$. By dyadic decomposition as in the previous case, we write $J_2$ as 
\begin{align*}
    \rho_L(\eta) J_2(t,\eta) &= \sum_{(N_1,N_2)\in    \mathcal A_L } J_2^{(L,N_1,N_2)}(t,\eta),  \\ 
    J_2^{(L,N_1,N_2)}(t,\eta)&= \frac{1}{t^2}\int_{\R^2} 
    \mathbf{m}_2^{(L,N_1,N_2)}(\eta,\sigma) \cdot  e^{-it\langle\sigma\rangle}\widehat{S_{N_1}xf}(t,\sigma)\overline{\widehat{S_{N_2}u}(t,\eta+\sigma)} d\sigma \\ 
    \mathbf{m}_2^{(L,N_1,N_2)}(\eta,\sigma)&=\nabla_\sigma\cdot \mathbf{m}(\eta,\sigma)\,\mathbf{m}(\eta,\sigma)\rho_L (\eta)\rho_{N_1} (\eta+\sigma)\rho_{N_2} (\sigma).
\end{align*}
A direct computation with \eqref{lowbound} yields that 
\begin{align}\label{cmbound2}
    C\big(\mathbf{m}_2^{(L,N_1,N_2)}\big) \les L^{-2}\max(N_1, N_2)^2 \min(N_1, N_2)^{13}.
\end{align}
Applying the operator inequality \eqref{eq:coif-1} with \eqref{cmbound2}, we obtain 
\begin{align*}
    &\sum_{(N_1,N_2)\in    \mathcal A_L} \left\|J_2^{(L,N_1,N_2)}(t)\right\|_{L^2} \\ 
    &\; \les |t|^{-2}   L^{-2} \sum_{(N_1,N_2)\in    \mathcal A_L}\max(N_1, N_2)^2 \min(N_1, N_2)^{13}
    \|S_{N_1}xf(t)\|_{L^2(\R^2)} \|S_{N_2}u(t)\|_{L^\infty(\R^2)} \\ 
    &\; \lesssim |t|^{-3}L^{-2} \sum_{(N_1,N_2)\in    \mathcal A_L} \max(N_1, N_2)^2 \min(N_1, N_2)^{13}N_1^{-10}N_2^{-7}\ve_1^2 
    \\ &\; \lesssim |t|^{-3}L^{-2}\ve_1^2,
\end{align*}
where we used that
\begin{align}\label{xf}
    \|xf(t)\|_{H^{10}(\R^2)}\lesssim \ep_1,
\end{align}
which holds under the a priori assumption, as explained in \eqref{x by x square}.

$J_3$ can be estimated similarly as $J_2$. We omit the proof.

Lastly, we consider $J_4$. We write 
\begin{align*}
    J_4= J_{4,1} + J_{4,2}
\end{align*}
where
\begin{align*}J_{4,1}(t,\eta)&=\frac{1}{t^2}\int_{\R^2}e^{it\langle \sigma+\eta\rangle } \mathbf{m}\otimes \mathbf{m}(\eta,\sigma)\cdot  \Big( \widehat{u}(t,\sigma)\overline{\widehat{x^2f}(t,\eta+\sigma})\Big) d\sigma \\
  &\qquad +\frac{1}{t^2}\int_{\R^2}e^{-it \langle\sigma\rangle } \mathbf{m}\otimes \mathbf{m}(\eta,\sigma)\cdot  \Big( \widehat{x^2f}(t,\sigma)\overline{\widehat{u}(t,\eta+\sigma})\Big) d\sigma, \\
    J_{4,2}(t,\eta)&=\frac{1}{t^2}\int_{\R^2}e^{it(\langle \sigma+\eta\rangle - \langle\sigma\rangle )} \mathbf{m}\otimes \mathbf{m}(\eta,\sigma)\cdot  \Big( \widehat{xf}(t,\sigma)\otimes\overline{\widehat{xf}(t,\eta+\sigma})\Big) d\sigma.
\end{align*}
For $J_{4,1}$, we only consider the first term. The other term can be dealt with analogously. As before, we perform dyadic decomposition 
\begin{align*}
    \chi_L(\eta) J_{4,1}(t,\eta) &= \sum_{(N_1,N_2)\in \mathcal{A}_L} J_{4,1}^{(L,N_1,N_2)}(t,\eta),  \\ 
    J_{4,1}^{(L,N_1,N_2)}(t,\eta)&= \frac{1}{t^2}\int_{\R^2} 
    \mathbf{m}_{4,1}^{(L,N_1,N_2)}(\eta,\sigma) \cdot  e^{-it\langle\sigma\rangle}\widehat{S_{N_1}x^2f}(t,\sigma)\overline{\widehat{S_{N_2}u}(t,\eta+\sigma)} d\sigma, \\ 
    \mathbf{m}_{4,1}^{(L,N_1,N_2)}(\eta,\sigma)&=\mathbf{m}\otimes \mathbf{m}(\eta,\sigma)\chi_L (\eta)\rho_{N_1} (\eta+\sigma)\rho_{N_2} (\sigma).
\end{align*}
A direct computation with \eqref{lowbound} yields that 
\begin{align}\label{cmbound41}
    C\big(\mathbf{m}_{4,1}^{(L,N_1,N_2)}\big) \les L^{-2}\max(N_1, N_2)^2 \min(N_1, N_2)^{12},
\end{align}
so, applying the operator inequality \eqref{eq:coif-1} with \eqref{cmbound41}, we obtain 
\begin{align*}
    &\sum_{(N_1,N_2)\in \mathcal{A}_L}   \left\| J_{4,1}^{(L,N_1,N_2)}(t) \right\|_{L^2(\R^2)} \\ 
    &\;\les |t|^{-2}  L^{-2} \sum_{(N_1,N_2)\in \mathcal{A}_L}   \max(N_1, N_2)^2 \min(N_1, N_2)^{12}
    \|S_{N_1}x^2f(t)\|_{L^2(\R^2)} \|S_{N_2}u(t)\|_{L^\infty(\R^2)} \\ 
    &\;\les |t|^{-2}  L^{-2} \sum_{(N_1,N_2)\in \mathcal{A}_L}  \max(N_1, N_2)^2 \min(N_1, N_2)^{12}  N_1^{-10} 
   N_2^{-7}\langle t\rangle^{-1}\ve_1^2 \\ 
    &\;\les \langle t\rangle^{-3}L^{-2}.
\end{align*}
For $J_{4,2}$, we integrate by parts once again using the relation \eqref{rel} to obtain 
\begin{align*}
    J_{4,2}(t,\eta)&= J_{4,2,1}(t,\eta) +J_{4,2,2}(t,\eta), \\ 
    J_{4,2,1}(t,\eta)&=\frac{i}{t^3}\int_{\R^2}e^{it(\langle \sigma+\eta\rangle - \langle\sigma\rangle )} \mathbf{m}\otimes \mathbf{m}\otimes \mathbf{m}(\eta,\sigma)\cdot \nabla\otimes \Big( \widehat{xf}(t,\sigma)\otimes\overline{\widehat{xf}(t,\eta+\sigma})\Big) d\sigma, \\ 
    J_{4,2,2}(t,\eta)&=\frac{i}{t^3}\int_{\R^2}e^{it(\langle \sigma+\eta\rangle - \langle\sigma\rangle )} \nabla_\sigma \cdot \Big( \mathbf{m}(\eta,\sigma) \, \mathbf{m}\otimes \mathbf{m}(\eta,\sigma) \Big)\cdot  \Big( \widehat{xf}(t,\sigma)\otimes\overline{\widehat{xf}(t,\eta+\sigma})\Big) d\sigma.
\end{align*}
Using the following pointwise bounds of the multiplier
\begin{align*}
    \left|  \mathbf{m}\otimes \mathbf{m}\otimes \mathbf{m}(\eta,\sigma) \right| 
    \lesssim |\eta|^{-3} \min(\langle \sigma\rangle, \langle \sigma+\eta\rangle)^{3}\max(\langle \sigma\rangle, \langle \sigma+\eta\rangle)^{6},
\end{align*}
and 
\begin{align*}
    \left| \nabla_\sigma \cdot \big( \mathbf{m}(\eta,\sigma) \, \mathbf{m}\otimes \mathbf{m}(\eta,\sigma) \big) \right| 
    \lesssim |\eta|^{-3} \min(\langle \sigma\rangle, \langle \sigma+\eta\rangle)^{3}\max(\langle \sigma\rangle, \langle \sigma+\eta\rangle)^{7},
\end{align*}
we obtain by the H\"older inequality
\begin{align*}
  \left\|   J_{4,2,1}(t) \right\|_{L^2(\R^2)} &\lesssim 
  |t|^{-3}\left\| |\eta|^{-3}\chi_L(\eta) \right\|_{L^2(\R^2)} 
  \| x^2 f(t)\|_{H^{10(\R^2)} } \| xf(t)\|_{H^{10}(\R^2)} \lesssim |t|^{-3}L^{-2}\ve_1^2, \end{align*} and 
  \begin{align*}\left\|   J_{4,2,2}(t) \right\|_{L^2(\R^2)} &\lesssim 
  |t|^{-3} \left\| |\eta|^{-3}\chi_L(\eta) \right\|_{L^2(\R^2)} 
  \| x f(t)\|_{H^{10}(\R^2) } \| xf(t)\|_{H^{10}(\R^2)}\lesssim |t|^{-3}L^{-2}\ve_1^2,
\end{align*} 
respectively.
\end{proof}



  \bigskip

\section{\label{sec:Weighted-Energy-estimate}Weighted Energy estimate : Proof of Proposition~\ref{Prop:Weighted estimates}}
In this section, we prove Proposition~\ref{Prop:Weighted estimates} and complete the proof of our main theorem.

\begin{proof}[\textbf{Proof of \eqref{weighted estimates 1}}]

Recall that under the a priori assumption \eqref{appen-apriori}, one has 
\begin{align*}
    \| u(t) \|_{W^{7,\infty}(\R^2)} \lesssim \langle t\rangle^{-1}\ve_1.
\end{align*}
Interpolating this with the conservation law \eqref{mass conservation}
\begin{align*}
   \| u(t) \|_{L^2(\R^2)} = \|u_0\|_{L^2(\R^2)},
\end{align*}
we obtain for $2\le p\le\infty$
\begin{align}\label{Lp timedecay}
   \| u(t) \|_{L^p(\R^2)} \lesssim \langle t\rangle^{-2(\frac12-\frac1p)} \ve_1.
\end{align}
Then, by using \eqref{eq:hls}, Leibnitz rule, and Hardy-Littlewood-Sobolev inequality, we estimate  
\begin{align*}
    \| N_\gamma (u,u,u)(t)\|_{H^s(\R^2)} 
    &\lesssim \| |x|^{-\gamma}\ast |u(t)|^2 \|_{L^\infty(\R^2)} \|u(t)\|_{H^s(\R^2)} + 
    \| |x|^{-\gamma}\ast |u(t)|^2 \|_{H_{\frac{4}{\gamma}}^s(\R^2)} \| u(t)\|_{L^{\frac{4}{2-\gamma}}(\R^2)} \\ 
    &\lesssim \|u(t)\|_{L^\infty(\R^2)}^{\gamma}\|u_0\|_{L^2(\R^2)}^{2-\gamma}\|u(t)\|_{H^s(\R^2)}
    +\|u(t)\|_{H^s(\R^2)}\|u(t)\|_{L^{\frac{4}{2-\gamma}}(\R^2)}^2 \\ &\lesssim \langle t\rangle^{-\gamma}\ep_1^3.
\end{align*}
\end{proof}

\begin{proof}[\textbf{Proof of \eqref{weighted estimates 3}}]
By the Plancherel's theorem, we have
\[
\|x^{2}e^{it\jp D}u\|_{H^{10}(\R^2)}\sim\|\langle\xi\rangle^{10}\mathcal{F}(x^{2}e^{it\jp D}u)\|_{L^{2}(\R^2)}\sim\left\Vert \langle\xi\rangle^{10}\nabla_{\xi}^{\,2}\widehat{f}\,\right\Vert _{L^{2}(\R^2)}.
\]
The Duhamel's formula \eqref{eq:duhamel} implies that $\nabla^{2}\widehat{f}$
can be represented by 
\begin{align*}
    \langle\xi\rangle^{10}\nabla^{2}\widehat{f}(t,\xi) & =\langle\xi\rangle^{10}\nabla^{2}\widehat{u_{0}}(\xi)+i\lam\sum_{j=1}^{4}\int_0^t\mathcal{I}_{j}(t',\xi)dt',
\end{align*}
where, by abusing the notation,
\begin{align*}
\mathcal{I}_{1}(t',\xi) & =\langle\xi\rangle^{10}\int_{\mathbb{R}^{2}}e^{it'\phi(\xi,\eta)}|\eta|^{-2+\gamma}\widehat{x^2f}(t',\xi-\eta)\mathcal{F}(|u(t')|^{2})(\eta)d\eta ,\nonumber \\
\mathcal{I}_{2}(t',\xi) & =2it'\langle\xi\rangle^{10}\int_{\mathbb{R}^{2}}e^{it'\phi(\xi,\eta)}|\eta|^{-2+\gamma}\nabla_{\xi}\phi(\xi,\eta)\cdot\widehat{xf}(t',\xi-\eta)\mathcal{F}(|u(t')|^{2})(\eta)d\eta ,\nonumber \\
\mathcal{I}_{3}(t',\xi) & =it'\langle\xi\rangle^{10}\int_{\mathbb{R}^{2}}\nabla_{\xi}\otimes\nabla_{\xi}\phi(\xi,\eta)(\xi,\eta)e^{it'\phi(\xi,\eta)}|\eta|^{-2+\gamma}\widehat{f}(t',\xi-\eta)\mathcal{F}(|u(t')|^{2})(\eta)d\eta ,\nonumber \\
\mathcal{I}_{4}(t',\xi) & =-(t')^{2}\langle\xi\rangle^{10}\int_{\mathbb{R}^{2}}\big( \nabla_{\xi}\phi\otimes\nabla_{\xi}\phi\big)(\xi,\eta)e^{it'\phi(\xi,\eta)}|\eta|^{-2+\gamma}\widehat{f}(t',\xi-\eta)\mathcal{F}(|u(t')|^{2})(\eta)d\eta, 
\end{align*}
where 
\begin{align} \label{eq:pm}\phi(\xi,\eta) & =\langle\xi\rangle-\langle\xi-\eta\rangle, \text{ and } 
 \nabla_{\xi}\phi(\xi,\eta)=\frac{\xi}{\langle\xi\rangle}-\frac{\xi-\eta}{\langle\xi-\eta\rangle}.
\end{align}
We prove that 
\[
\sum_{j=1}^{4}\normo{\mathcal{I}_{j}(t',\xi)}_{L_{\xi}^{2}(\R^{2})}\les\langle t'\rangle^{-1-\frac{\gamma-1}{3}}\ve_{1}^{3}.
\]

\smallskip

\noindent \uline{Estimates for \mbox{$\mathcal{I}_{1}$}}. 
We write 
\begin{align*}
    \normo{\mathcal{I}_{1}(t',\xi)}_{L_{\xi}^{2}(\R^{2})} 
    \le \left\| \mathcal{I}_{1,1}(t',\xi)\right\|_{L_{\xi}^{2}(\R^{2})} + \left\| \mathcal{I}_{1,2}(t',\xi)\right\|_{L_{\xi}^{2}(\R^{2})},
\end{align*}
where 
\begin{align*}
    \mathcal{I}_{1,1}(t',\xi) & = \int_{\mathbb{R}^{2}}e^{it'{\phi}(\xi,\eta)} |\eta|^{-2+\gamma}\langle \xi-\eta\rangle^{10}\widehat{x^{2}f}(t',\xi-\eta)\mathcal{F}(|u(t')|^{2})(\eta)\,d\eta, 
    \\ 
    \mathcal{I}_{1,2}(t',\xi) & = \int_{\mathbb{R}^{2}}e^{it'{\phi}(\xi,\eta)}\big(\langle\xi\rangle^{10} - \langle \xi-\eta\rangle^{10}\big) |\eta|^{-2+\gamma}\widehat{x^{2}f}(t',\xi-\eta)\mathcal{F}(|u(t')|^{2})(\eta)\,d\eta.
\end{align*}
Using \eqref{eq:hls} and time decay estimates \eqref{Lp timedecay}, the first term is bounded by 
\begin{align*}
    \left\| \mathcal{I}_{1,1}(t',\xi)\right\|_{L_{\xi}^{2}(\R^{2})}
    &\lesssim 
    \Big\| \big( |x|^{-\gamma}\ast|u(t')|^2 \big)\langle D\rangle^{10} (x^{2}f)\left(t'\right)  \Big\|_{L^2(\R^{2})} \\
    &\lesssim \left\| |x|^{-\gamma}\ast|u(t')|^2 \right\|_{L^\infty(\R^{2})}\|x^{2}f(t')\|_{H^{10}(\R^{2})} \\ 
    &\lesssim \|u(t')\|_{L^\infty(\R^{2})}^{\gamma}\|u(t')\|_{L^2(\R^{2})}^{2-\gamma}\|x^{2}f(t')\|_{H^{10}(\R^{2})} \lesssim \langle t'\rangle^{-\gamma}\ve_1^3.
\end{align*}
Next, for $I_{1,2}$, we observe that the singularity near the origin, $|\eta|^{-2+\gamma}$, is removed by the multiplier as follows:
\begin{align}\label{remove singularity}
    \big(\langle\xi\rangle^{10} - \langle \xi-\eta\rangle^{10}\big) |\eta|^{-2+\gamma}
    \lesssim  |\eta|^{\gamma-1}\max(\langle\xi\rangle^{9}, \langle \xi-\eta\rangle^{9}). 
\end{align}
We perform dyadic decomposition 
\begin{align*}
   \mathcal{I}_{1,2}(t',\xi) & = 
    \sum_{N_1,N_2\in 2^{\N\cup\{0\}}}
    \int_{\mathbb{R}^{2}}e^{it'{\phi}(\xi,\eta)}\big(\langle\xi\rangle^{10} - \langle \xi-\eta\rangle^{10}\big) |\eta|^{-2+\gamma}
    \widehat{S_{N_1}x^{2}f}(t',\xi-\eta)\widehat{S_{N_2}|u(t')|^{2}}(\eta)\,d\eta,
\end{align*}
and obtain that 
\begin{align*}
 &\left\| \mathcal{I}_{1,2}(t',\xi)\right\|_{L_{\xi}^{2}(\R^{2})}\lesssim  \\ 
 &\qquad \sum_{N_0,N_1,N_2\in 2^{\N\cup\{0\}}} \left\|  \int_{\mathbb{R}^{2}}e^{it'{\phi}(\xi,\eta)}  m_1^{(N_0,N_1,N_2)}(\xi,\eta)
 \widehat{S_{N_1} x^{2}f}(t',\xi-\eta)\widehat{S_{N_2}|u(t')|^{2}}(\eta)\,d\eta \right\|_{L_{\xi}^{2}(\R^{2})},
\end{align*}
where 
\begin{align*}
    m_1^{(N_0,N_1,N_2)}(\xi,\eta):=\big(\langle\xi\rangle^{10} - \langle \xi-\eta\rangle^{10}\big) |\eta|^{-2+\gamma}
    \rho_{N_0}(\xi)\rho_{N_1}(\xi-\eta)\rho_{N_2}(\eta).
\end{align*}
Here, we observe that the summation actually runs over those indices $(N_0,N_{1},N_2)$ in the following set 
\begin{align*}
 \mathcal{A}:= \left\{ (N_0,N_{1},N_2)\in (2^{\N\cup\{0\}})^3 \; | \; N_0\les N_1\sim N_2 \text{ or } N_0\sim \max(N_1,N_2)\right\}.
\end{align*}
A direct computation with \eqref{remove singularity} yields that 
\begin{align*}
    C\big(m_1^{(N_0,N_1,N_2)}\big) \lesssim  \max(N_0,N_1)^{9}N_2^{\,-1+\gamma},
\end{align*}
thus, by using the operator inequality \eqref{eq:coif-1} with the above bound, we have 
\begin{align*}
    \left\| \mathcal{I}_{1,2}(t',\xi)\right\|_{_{L_{\xi}^{2}(\R^{2})}}&\lesssim \sum_{(N_0,N_1,N_2)\in \mathcal{A}}  \max(N_0,N_1)^{9}N_2^{\,-1+\gamma} \|S_{N_1} x^{2}f(t')\|_{L^2(\R^2)}\|S_{N_2} |u(t')|^{2}\|_{L^\infty(\R^2)}.
\end{align*}
Then, by applying \eqref{eq:norm-infty} to the quadratic term, we bound the sum over $N_2\le \langle t'\rangle^{\frac{2}{s-9}}$ by 
\begin{align*}
    \langle t'\rangle^{-2} \sum_{\substack{(N_0,N_1,N_2)\in \mathcal{A} \\ N_2\le \langle t'\rangle^{\frac{2}{s-9}} } }  \max(N_0,N_1)^{9}N_2^{-1+\gamma} N_1^{-10}N_2^{-7}  \ep_1^3
&\lesssim \langle t'\rangle^{-2+\frac{2}{s-9}(\gamma+1)}\ve_1^3,
\end{align*}
which is bounded by $\langle t'\rangle^{-\frac{\gamma-1}{3}}\ve_1^3$, since $s\ge30$.
On the other hand, we use \eqref{eq:norm-infty 2} to bound the remaining contribution by
\begin{align*}
    \sum_{\substack{(N_0,N_1,N_2)\in \mathcal{A} \\ N_2 \ge \langle t'\rangle^{\frac{2}{s-9}}} }  \max(N_0,N_1)^{9} N_1^{-10} N_2^{1+\gamma-s}  \ve_1^3
    \lesssim \langle t'\rangle^{-2+\frac{2}{s-9}(\gamma+1)}\ve_1^3.
\end{align*}



\smallskip

\noindent \uline{Estimates for \mbox{$\mathcal{I}_{2}$}}. We decompose the frequency variables $\xi,\xi-\eta$ into the dyadic pieces $|\xi|\sim N_0$ and $|\xi-\eta|\sim N_1$ for $N_{0},N_{1}\in2^{\N\cup\{0\}}$, respectively.
On the other hand, for $\eta$ variable associated to the potential, we decompose not only high frequency but also the low frequency part near the origin, i.e., $|\eta|\sim L$ for $L\in2^{\Z}$. Then, we write 
\begin{align*}
    \mathcal{I}_{2}(t',\xi)=\sum_{N_{0},N_{1}\in 2^{\N\cup\{0\}},\, L\in 2^{\Z}}\mathcal{I}_{2}^{(N_{0},N_{1},L)}(t',\xi),
\end{align*}
where 
\begin{align*}
 \mathcal{I}_{2}^{(N_0,N_{1},L)}(t',\xi) =it'\int_{\mathbb{R}^{2}}e^{it'\phi(\xi,\eta)}\mathbf{m}_2^{(N_0,N_{1},L)}(\xi,\eta)\cdot\widehat{S_{N_1}xf}(t',\xi-\eta)\widehat{P_{L}(|u(t')|^{2})}(\eta)\,d\eta,
\end{align*}
where 
\begin{align*}
    \mathbf{m}_2^{(N_0,N_{1},L)}(\xi,\eta) 
    = \langle \xi\rangle^{10} |\eta|^{-2+\gamma} \, \nabla_{\xi}\phi(\xi,\eta)\rho_{N}(\xi)\rho_{N_1}(\xi-\eta)\chi_{L}(\eta).
\end{align*}
Then, 
\begin{align*}
    \left\| \mathcal{I}_{2}(t',\xi)\right\|_{L_{\xi}^2(\R^2)} 
    \lesssim \sum_{(N_0,N_1,L)\in \tilde {\mathcal A}} \left\| \mathcal{I}^{(N_{0},N_{1},L)}_{2}(t',\xi) \right\|_{L_{\xi}^2(\R^2)},
\end{align*}
where 
\begin{align}\label{tilde A}
    \tilde {\mathcal A} = \left\{ (N_0,N_{1},L)\in (2^{\N\cup\{0\}})^2\times 2^{\Z} \; | \; L\les N_1\sim N_2 \text{ or } L\sim \max(N_0,N_1)
    \right\}.
\end{align}
From \eqref{lowbound}, one can readily verify that
\begin{align}\label{boundm2}
 C\big( \mathbf{m}_2^{(N_0,N_{1},L)} \big)\les N_0^{10}L^{-1+\gamma}\max( N_0,  N_1)^{-1}.
\end{align} 
Then, by the operator inequality \eqref{eq:coif-1} with the above bound, we obtain   
\begin{align}\begin{aligned}\label{I2:CM}
&\left\| \mathcal{I}_{2}^{(N_0,N_{1},L)}(t',\xi) \right\|_{L_\xi^2(\R^2)} \lesssim  \\ 
&\qquad |t'|  N_0^{10}L^{-1+\gamma}\max( N_0,  N_1)^{-1}\normo{S_{N_{1}}{xf}(t')\,}_{L^{2}(\R^2)}\normo{P_{L}|u(t')|^{2}}_{L^{\infty}(\R^2)}.
\end{aligned}\end{align}
We observe from \eqref{xf} that under the a priori assumption
$$\normo{S_{N_{1}}{xf}(t')\,}_{L^{2}} \lesssim N_1^{-10}\ep_1.$$
Using this and \eqref{eq:norm-infty 2}, we bound the summation over $L\le \langle t'\rangle^{-1}$ by 
\begin{align*}
    \sum_{\substack{(N_0,N_1,L)\in \tilde {\mathcal A} \\ L\,\lesssim \,\langle t'\rangle^{-1}}} \left\| \mathcal{I}_{2}^{(N_0,N_{1},L)}(t',\xi) \right\|_{L_\xi^2(\R^2)} 
    &\lesssim |t'|\sum_{\substack{(N_0,N_1,L)\in \tilde {\mathcal A} \\ L\,\lesssim \,\langle t'\rangle^{-1}}} N_0^{10}L^{1+\gamma}\max(N_0,  N_1)^{-1} N_1^{-10}\ve_1^3\\ 
    &\lesssim \langle t'\rangle^{-\gamma}\ve_1^3,
\end{align*}
the summation over $ \langle t'\rangle^{-1} \le L \le  \langle t'\rangle^{\frac{2}{s-3}}$ by 
\begin{align*}
    &\sum_{\substack{(N_0,N_1,L)\in \tilde {\mathcal A} \\ \langle t'\rangle^{-1}\,\lesssim L\,\lesssim \, \langle t'\rangle^{\frac{3}{s-3}}}} \left\| \mathcal{I}_{2}^{(N_0,N_{1},L)}(t',\xi) \right\|_{L_\xi^2(\R^2)} \\
    &\quad\lesssim   \sum_{\substack{(N_0,N_1,L)\in \tilde {\mathcal A} \\ \langle t'\rangle^{-1}\,\lesssim L\,\lesssim \, \langle t'\rangle^{\frac{2}{s-3}}}} |t'|^{-2} N_0^{10}L^{-2+\gamma}\max( N_0,  N_1)^{-1}  N_1^{-10}\ve_1^3\\ 
    &\quad\lesssim \langle t'\rangle^{-\gamma}\ve_1^3 + \langle t'\rangle^{-2+\frac{2}{s-3}(7+\gamma)}\ve_1^3,
\end{align*}
and finally the summation over $L\ge \langle t'\rangle^{\frac{2}{s-3}}$ by 
\begin{align*}
    \sum_{\substack{(N_0,N_1,L)\in \tilde {\mathcal A} \\ L\,\ge  \, \langle t'\rangle^{\frac{2}{s-3}}}} \left\| \mathcal{I}_{2}^{(N_0,N_{1},L)}(t',\xi) \right\|_{L_\xi^2(\R^2)} 
    &\lesssim |t'|\sum_{\substack{(N_0,N_1,L)\in \tilde {\mathcal A} \\ L\,\ge  \, \langle t'\rangle^{\frac{2}{s-3}}}}  N_0^{10}L^{1+\gamma-s}\max(N_0,  N_1)^{-1} N_1^{-10} \\ 
    &\lesssim   \langle t'\rangle^{-2+\frac{2}{s-3}(7+\gamma)}\ve_1^3,
\end{align*}
which is bounded by $\langle t'\rangle^{-\frac{\gamma-1}{3}}\ve_1^3$, since $s\ge30$.

\smallskip

\noindent \uline{Estimates for \mbox{$\mathcal{I}_{3}$}}.
$\mathcal{I}_{3}$ can be handled in an anolugous manner to $\mathcal{I}_{2}$.  
Indeed, one finds that the multiplier $\nabla_{\xi}^2\phi(\xi,\eta)$ in $\mathcal{I}_{3}$ verifies the smaller bound than the one given in \eqref{lowbound} satisfied by the multiplier $\nabla_{\xi}\phi(\xi,\eta)$ in $\mathcal{I}_{2}$. 
More precisely, the following bound holds
\begin{align*}
    \left| \nabla_{\xi}^2\phi(\xi,\eta)\right| 
    \les  \frac{|\eta|}{\max(\langle\xi\rangle,\langle\xi-\eta\rangle)^2}.
\end{align*}
Thus, if we let 
\begin{align*}
    \mathbf{m}_3^{(N_{0},N_{1},L)}(\xi,\eta) 
    =\langle \xi \rangle^{10}  |\eta|^{-2+\gamma}  \nabla_{\xi}^2\phi(\xi,\eta)\rho_{N_{0}}(\xi)\chi_{L}(\eta)\rho_{N_{1}}(\xi-\eta) ,
\end{align*}
one can verify that 
\begin{align*}
 C(\mathbf{m}_3^{(N_{0},N_{1},L)}) &\les  N_0^{10}\max(N_0, N_1)^{-2}L^{-1+\gamma}.
\end{align*} 
Applying these bounds into \eqref{I2:CM} instead of \eqref{boundm2}, one can obtain the desired bounds.

\smallskip
\noindent \uline{Estimates for \mbox{$\mathcal{I}_{4}$}}.
It remains to estimate the main case $\mathcal{I}_{4}$. 
As before, we decompose 
\begin{align*}
    \mathcal{I}_{4 }(t',\xi)=\sum_{N_{0},N_{1}\in (2^{\N})^2,\, L\in2^{\Z}}\mathcal{I}_4^{(N_{0},N_{1},L)}(t',\xi),
\end{align*}where 
\begin{align*}
    \mathcal{I}_4^{(N_{0},N_{1},L)}(t',\xi) =-(t')^2\int_{\mathbb{R}^{2}}\mathbf{m}_4^{(N_{0},N_{1},L)}(\xi,\eta)e^{it'\phi(\xi,\eta)}\widehat{S_{N_1}f}(t',\xi-\eta)\widehat{P_{L}(|u(t')|^{2})}(\eta)\,d\eta,
\end{align*}
where 
\begin{align*}
    \mathbf{m}_4^{(N_{0},N_{1},L)}(\xi,\eta) 
=\langle \xi\rangle^{10}   |\eta|^{-2+\gamma} \big( \nabla_{\xi}\phi(\xi,\eta)\big)^2\rho_{N_{0}}(\xi)\rho_{N_{1}}(\xi-\eta)\chi_{L}(\eta).
\end{align*}
Then, 
\begin{align*}
    \left\|\mathcal{I}_{4}(t',\xi)\right\|_{L_{\xi}^2(\R^2)} 
    \lesssim \sum_{(N_0,N_1,L)\in \tilde {\mathcal A}} \left\| \mathcal{I}^{(N_{0},N_{1},L)}_{4}(t',\xi) \right\|_{L_{\xi}^2(\R^2)},
\end{align*}
where $\tilde {\mathcal A}$ is given in \eqref{tilde A}.
From \eqref{lowbound}, one can readily verify that  
\begin{align*}
     C\big(\mathbf{m}_4^{(N_0,N_1,L)}\big)\les  L^{\gamma}N_0^{10}\max( N_0,  N_1)^{-2},
\end{align*} 
which gives that 
\begin{align*}
    \left\| \mathcal{I}_4^{(N_0,N_1,L)}(t',\xi) \right\|_{L_{\xi}^2(\R^2)}  \lesssim |t'|^2L^{\gamma} N_0^{10}\max(N_0, N_1)^{-2}
    \| S_{N_1}u(t')\|_{L^\infty(\R^2)}
    \| P_{L}(|u(t')|^{2})\|_{L^2(\R^2)}.
\end{align*}
Using \eqref{eq:norm-two}, we estimate the summation over $L\le \langle t'\rangle^{-1}$ 
\begin{align*}
\sum_{\substack{(N_0,N_1,L)\in \tilde {\mathcal A} \\ L\le \langle t'\rangle^{-1}} } \left\| \mathcal{I}^{(N_{0},N_{1},L)}_{4}(t',\xi) \right\|_{L_{\xi}^2(\R^2)}
\lesssim \sum_{\substack{(N_0,N_1,L)\in \tilde {\mathcal A} \\ L\le \langle t'\rangle^{-1}} }|t'| L^{\gamma+1} N_0^{10}\max( N_0, N_1)^{-2}N_1^{-7}, 
\end{align*}
the summation over $\langle t'\rangle^{-1} \le L \le \langle t'\rangle ^{\frac{3}{s-3}}$
\begin{align*}
    &\sum_{\substack{(N_0,N_1,L)\in \tilde {\mathcal A} \\ \langle t'\rangle^{-1} \le L \le  \langle t'\rangle ^{\frac{3}{s-3}}} } \left\| \mathcal{I}^{(N_{0},N_{1},L)}_{4}(t',\xi) \right\|_{L_{\xi}^2(\R^2)} \\
    &\quad \lesssim \sum_{\substack{(N_0,N_1,L)\in \tilde {\mathcal A} \\ \langle t'\rangle^{-1} \le L \le  \langle t'\rangle ^{\frac{3}{s-3}}} } \langle t'\rangle^{-2} L^{\gamma-2}N_0^{10}\max(N_0,N_1)^{-2}N_1^{-7}\ve_1^3 \\ 
    &\quad \lesssim \langle t'\rangle^{-2}\ve_1^3 + \langle t'\rangle^{-2+\frac{3(6+\gamma)}{s-3}}\ve_1^3,
    \end{align*}
and finally the summation over $L \ge \langle t'\rangle ^{\frac{3}{s-3}}$
    \begin{align*}
        \sum_{\substack{(N_0,N_1,L)\in \tilde {\mathcal A} \\ L \ge \langle t'\rangle ^{\frac{3}{s-3}} } } \left\| \mathcal{I}^{(N_{0},N_{1},L)}_{4}(t',\xi) \right\|_{L_{\xi}^2(\R^2)}
        &\lesssim \sum_{\substack{(N_0,N_1,L)\in \tilde {\mathcal A} \\ L \ge \langle t'\rangle ^{\frac{3}{s-3}} } } |t'| L^{\gamma+1-s} \langle N_0\rangle^{10}\max(N_0,N_1)^{-2}  N_1^{-7}\ve_1^3 \\ 
        &\lesssim \langle t'\rangle^{-2+\frac{3}{s-3}(\gamma+6)}\ve_1^3,
\end{align*}
which is bounded by $\langle t'\rangle^{-\frac{\gamma-1}{3}}\ve_1^3$, since $s\ge30$.
\end{proof}

 \bibliographystyle{plain}
\bibliography{ReferencesCKLY}

\smallskip

\end{document}